\documentclass[reqno,11pt]{amsart}
\DeclareMathOperator*{\esssup}{\mathrm{ess\,sup}}

\usepackage{amssymb,amsthm}
\usepackage{amsmath}
\usepackage{amsfonts}
\usepackage{dsfont}

\usepackage[a4paper,  margin=3.3cm]{geometry}

\usepackage[active]{srcltx}
\makeatletter\@addtoreset{equation}{section}\makeatother

\newtheorem{theorem}{Theorem}[section]

\newtheorem{lemma}[theorem]{Lemma}

\newtheorem{proposition}[theorem]{Proposition}
\newtheorem{assumption}[theorem]{Assumption}
\newtheorem{definition}[theorem]{Definition}
\newtheorem{remark}[theorem]{Remark}
\numberwithin{equation}{section}

\title{Self-regulation in continuum population models}

\author{Yuri  Kondratiev}
\address{Fakut\"at f\"ur Mathematik, Universit\"at Bielefeld, Bielefeld D-33615, Germany and Interdisciplinary Center
for Complex Systems, Dragomanov University, Kyiv, Ukraine}
\email{kondrat@math.uni-bielefeld.de}

\author{ Yuri  Kozitsky}

\address{Instytut Matematyki, Uniwersytet Marii Curie-Sk{\l}odowskiej, 20-031 Lublin, Poland}
\email{jkozi@hektor.umcs.lublin.pl}

\begin{document}


\begin{abstract}

We study the Markov dynamics of an infinite birth-and-death system
of point entities placed in $\mathds{R}^d$, in which the
constituents disperse and die, also due to competition. Assuming
that the dispersal and competition kernels are just continuous and
integrable we prove that the evolution of states of this model
preserves their sub-Poissonicity, and hence the local
self-regulation (suppression of clustering) takes place. Upper
bounds for the correlation functions of all orders are also obtained
for both long and short dispersals, and for all values of the
intrinsic mortality rate $m\geq 0$.

\end{abstract}

\maketitle

\section{Introduction}

\subsection{The setup}

The aim of the present work is to contribute to the development of
the mathematical theory of large systems of living entities, which
is a challenging task of modern applied  mathematics
\cite{BB,B,Neuhauser}. Within this task there is the description of
the dynamics of individual-based models in which communities of
entities appear as configurations of points in some continuous
habitat, see \cite{BP1,Mu,Ova,O}. In particular, these can be
birth-and-death models, in which the dynamics amounts to the
appearance (birth) and disappearance (death) of the constituents.
The fact that the disappearance of a given entity is related to its
interaction with the existing community is interpreted as
\emph{competition} between the entities.

In the simplest birth-and-death models, the system is finite and the
state space is $\mathds{N}_0 :=\mathds{N}\cup \{0\}$. That is, in
state $n\in \mathds{N}_0$ the system consists of $n$ entities, which
is the complete characterization of the state. Then the only
observed result of the trade-off between the appearance and
disappearance is the dynamics of the number of entities in the whole
population. The theory of such models goes back to works by A.
Kolmogorov and W. Feller, see \cite[Chapter XVII]{Feller} and, e.g.,
\cite{Banas,Ri}, for a more recent account of the related concepts
and results. Therein the time evolution of the probability of having
$n$ entities is obtained by solving the Kolmogorov equation  with a
tridiagonal infinite matrix on the right-hand side. Its entries are
expressed in terms of the birth and death rates $\lambda_n$ and
$\mu_n$, respectively. If the increase of $\lambda_n$ and $\mu_n$ is
controlled by affine functions of $n$, the solution of the
Kolmogorov equation in given by a stochastic semigroup, see, e.g.,
\cite{Banas} and the literature quoted in this work. However, if
$\lambda_n$ and $\mu_n$$\lambda_n$ increase faster than $n$, this is
no more the case. For infinite systems, the very definition of the
Kolmogorov equation gets problematic as the mentioned birth and
death rates get infinite. Usually such systems are considered in
some spatial habitat and the parameters that describe the
interactions between the entities are space-dependent. Then then
along with the global characteristics the system acquires a local
time-dependent structure. The mentioned trade-off may affect this
structure with or without affecting the global dynamics of the
population. This relates also to the models with traits other than
the spatial position.

In this work we continue dealing with the model introduced in
\cite{BP1,BP3,Mu}. Here the spatial habitat is the Euclidean space
$\mathds{R}^d$, $d\geq 1$, equipped with the usual outfit of
mathematical structures. Then the phase space is the set $\Gamma$ of
all subsets $\gamma \subset \mathds{R}^d$ such that the set
$\gamma_\Lambda:=\gamma\cap\Lambda$ is finite whenever $\Lambda
\subset \mathds{R}^d$ is compact. For each such $\Lambda$, one
defines the counting map $\Gamma \ni \gamma \mapsto |\gamma_\Lambda|
:= \#\{\gamma\cap\Lambda\}$, where the latter denotes cardinality.
Thereby, one introduces the subsets $\Gamma^{\Lambda,n}:=\{ \gamma
\in \Gamma : |\gamma_\Lambda| = n\}$, $n \in \mathds{N}_0$, and
equips $\Gamma$ with the $\sigma$-field generated by all such
$\Gamma^{\Lambda,n}$. This allows for considering probability
measures on $\Gamma$ as states of the system. Among them there are
Poissonian states in which the entities are independently
distributed over $\mathds{R}^d$,  see \cite[Chapter 2]{Kingman}.
They may serve as reference states for studying correlations between
the positions of the entities. For the homogeneous Poisson measure
$\pi_\varkappa$ with density $\varkappa
>0$ and every compact $\Lambda$, one has
\begin{equation}
  \label{J1}
\pi_\varkappa (\Gamma^{\Lambda,n}) = \left(\varkappa {\rm
V}(\Lambda)\right)^n \exp\left( - \varkappa {\rm V}(\Lambda)
\right)/ n!, \qquad n \in \mathds{N}_0,
\end{equation}
where ${\rm V}(\Lambda)$ denotes Lebesgue's measure (volume) of
$\Lambda$. Let us agree to call a state, $\mu$,
\emph{sub-Poissonian} (cf. Definition \ref{J1df} and Remark
\ref{I1rk} below) if, for each compact $\Lambda \subset
\mathds{R}^d$, the following holds
\begin{equation}
  \label{J2}
\forall n\in \mathds{N}_0 \qquad \mu(\Gamma^{\Lambda,n}) \leq
C_\Lambda \varkappa_\Lambda^n/n!,
\end{equation}
with some positive constants $C_\Lambda$ and $\varkappa_\Lambda$. By
the virtue of this definition, the sub-Poissonian states are
characterized by the lack of \emph{heavy tails} or
\emph{clustering}. That is, the entities in such a state are either
independent in taking their positions or `prefer' to stay away of
each other.

The counting map $\Gamma \ni \gamma \mapsto |\gamma|$ can also
be defined for $\Lambda = \mathds{R}^d$. Then the set of \emph{finite}
configurations
\begin{equation}
 \label{J3a}
 \Gamma_0:= \bigcup_{n\in \mathds{N}_0}\{
\gamma \in \Gamma: |\gamma|=n\}
\end{equation}
is clearly measurable. In a state with the property
$\mu(\Gamma_0)=1$, the system  is ($\mu$-almost surely)
\emph{finite}. By (\ref{J1}) $\pi_\varkappa(\Gamma_0) =0$, hence the
system in state $\pi_\varkappa$ is infinite in the same sense. A
nonhomogeneous Poisson measure $\pi_\varrho$, characterized by
density $\varrho:\mathds{R}^d\to [0,+\infty)$, satisfies (\ref{J1})
with $\varkappa {\rm V}(\Lambda)$ replaced by $\int_\Lambda
\varrho(x) dx$. Then either $\pi_\varrho(\Gamma_0) =1$ or
$\pi_\varrho(\Gamma_0) =0$, depending on whether or not $\varrho$ is
globally integrable. The use of infinite configurations for modeling
large finite populations is as a rule justified, see, e.g.,
\cite{Cox}, by the argument that in such a way one gets rid of the
boundary and size effects. Note that a finite system with dispersal
-- like the one specified in (\ref{L}) and (\ref{J4}) below -- being
placed in a noncompact habitat always disperse to fill its empty
parts, and thus is \emph{developing}. Infinite configurations are supposed
to model \emph{developed} populations. In this work, we shall
consider infinite systems and hence deal with states $\mu$ such that
$\mu(\Gamma_0)=0$.

\subsection{The article overview}

To characterize  states on $\Gamma$ one employs {\it observables} --
appropriate functions $F:\Gamma \rightarrow \mathds{R}$. Their
evolution is obtained from the Kolmogorov equation
\begin{equation}
 \label{R2}
\frac{d}{dt} F_t = L F_t , \qquad F_t|_{t=0} = F_0, \qquad t>0,
\end{equation}
where the generator $L$ specifies the model. The states' evolution
is then obtained from the Fokker--Planck equation
\begin{equation}
 \label{R1}
\frac{d}{dt} \mu_t = L^* \mu_t, \qquad \mu_t|_{t=0} = \mu_0,
\end{equation}
related to that in (\ref{R2}) by the duality $\mu_t(F_0) =
\mu_0(F_t)$, where
\[
\mu(F) := \int_{\Gamma} F(\gamma) \mu(d \gamma).
\]
The model discussed in this work is specified by the following
\begin{eqnarray}
 \label{L}
 \left(L F \right)(\gamma) & = & \sum_{x\in \gamma} E^{-} (x, \gamma \setminus x) \left[F(\gamma \setminus x) - F(\gamma) \right]\\[.2cm]
 & + & \int_{\mathds{R}^d} E^{+} (x, \gamma ) \left[F(\gamma \cup x) - F(\gamma) \right]dx, \nonumber
\end{eqnarray}
where $E^{+} (x, \gamma )$ and $E^{-} (x, \gamma )$ are
state-dependent birth  and death rates, respectively. They have the
following forms
\begin{equation}
 \label{J4}
E^{+} (x, \gamma ) = \sum_{y\in \gamma} a^{+} (x-y),
\end{equation}
\begin{equation}
 \label{J6}
 E^{-} (x, \gamma ) =  m + \sum_{y\in \gamma} a^{-} (x-y),
\end{equation}
where $a^{+}\geq 0$ and $a^{-}\geq 0$ are the \emph{dispersal} and
\emph{competition kernels}, respectively, $m\geq 0$  is the
intrinsic mortality rate.  This model was introduced in
\cite{BP1,BP3,Mu}. Its recent study can be found in
\cite{FKKK,KK,K}, see also older works \cite{Dima,DimaN2}. For the
kernels $a^{\pm}$, one has the following possibilities:
\begin{itemize}
  \item[(i)] (\emph{short dispersal}) there exists $\theta>0$ such that $a^{-} (x) \geq \theta
  a^{+}(x)$ for all $x\in
  \mathds {R}^d$;
\item[(ii)] (\emph{long dispersal}) for each $\theta>0$, there exists  $x\in
  \mathds {R}^d$ such that $a^{-} (x) < \theta
  a^{+}(x)$.
\end{itemize}
In case (i), $a^{+}$ decays faster than $a^{-}$, and hence each
daughter entity can `kill' her mother as well as can be `killed' by
her. Such models are usually employed to described the dynamics of
cell communities, see \cite{DimaRR}, where the dispersal is just the
cell division. An instance of the short dispersal is given by
$a^{+}$ with finite range, i.e., $a^{+}(x)\equiv 0$ for all $|x|
\geq r$, and $a^{-}(x)>0$ for such $x$. In case (ii), $a^{-}$ decays
faster than $a^{+}$, and hence some of the offsprings can be out of
reach of their parents. Models of this kind can be adequate, e.g.,
in plant ecology with the long-range dispersal of seeds, cf.
\cite{Ova}.

In this article, the model parameters are supposed to satisfy the
following.
\begin{assumption}
  \label{Ass1}
The kernels $a^{\pm}$ in (\ref{J4}) and (\ref{J6}) are continuous
and belong to $L^1 (\mathds{R}^d) \cap L^\infty (\mathds{R}^d)$.
\end{assumption}
According to this we set
\begin{gather}
  \label{J6a}
\langle a^{\pm}\rangle = \int_{\mathds{R}^d} a^{\pm } (x ) dx,
\qquad \|a^{\pm}\| = \sup_{x\in \mathds{R}^d} a^{\pm}(x) .
\end{gather}
For the model with $E^{+}$ as in (\ref{J4}), in \cite{KK} we have
constructed the evolution of states $\mu_0 \mapsto \mu_t$, $t>0$,
which preserves the sub-Poissonicity under a certain condition on
the dispersal and competition kernels. In this work, by means of the
result proved in Lemma \ref{J1lm} we eliminate this restriction and
prove that the local self-regulation in this model occurs (Theorem
\ref{1tm}) if $a^{\pm}$ satisfy just a minimum set of assumptions,
see Assumption \ref{Ass1} and also the corresponding comments in subsection 2.3. For
the migration model specified in (\ref{J5}) and (\ref{J6}), we prove
that the evolution of states $\mu_0 \mapsto \mu_t$, $t>0$, exists
(Theorem \ref{J2lm}) and is such that $\mu_t (N_\Lambda^n) \leq C^{(n)}_\Lambda$ for each $t>0$.
We do this as follows. First, assuming the existence of the evolution
$\mu_0 \mapsto \mu_t$, we prove that it is characterized by the
global self-regulation as just mentioned, see Theorem \ref{2tm} and Lemma \ref{J3lm}.
Then we prove the existence of this evolution.

The structure of the article is as follows. In Section 2, we
introduce the necessary technicalities and then formulate the
results: Theorems \ref{1tm} and \ref{2tm}. Thereafter, we make a
number of comments to these results and compare them with the facts
known for similar objects. In Section 3, we present the proof of the
both mentioned statements assuming the existence of the evolution of
states in the migration model. In Section 4, we prove the latter
fact.

\section{Preliminaries and the Results}

By $\mathcal{B}(\mathds{R})$ we denote the sets of all Borel subsets
of $\mathds{R}$. The configuration space $\Gamma$ is equipped with
the vague topology, see \cite{Albev,Tobi}, and thus with the
corresponding Borel $\sigma$-field $\mathcal{B}(\Gamma)$, which
makes it a standard Borel space. Note that $\mathcal{B}(\Gamma)$ is
exactly the $\sigma$-field generated by the sets
$\Gamma^{\Lambda,n}$, mentioned in Introduction.  By
$\mathcal{P}(\Gamma)$ we denote the set of all  probability measures
on $(\Gamma, \mathcal{B}(\Gamma))$.

\subsection{Correlation functions}

Like in \cite{Dima,FKKK,KK}, the evolution of states will be
described by means of correlation functions. To explain the essence
of this approach let us consider the set of all compactly supported
continuous functions $\theta:\mathbb{R}^d\to (-1,0]$. For a state,
$\mu$, its {\it Bogoliubov} functional, cf. \cite{TobiJoao}, is
\begin{equation}
  \label{I1}
B_\mu (\theta) = \int_{\Gamma} \prod_{x\in \gamma} ( 1 + \theta (x))
\mu( d \gamma),
\end{equation}
with $\theta$ running through the mentioned set of functions. For
the homogeneous Poisson measure $\pi_\varkappa$,  it takes the form
\begin{equation*}
B_{\pi_\varkappa} (\theta) = \exp\left(\varkappa
\int_{\mathbb{R}^d}\theta (x) d x \right).
\end{equation*}
\begin{definition}
  \label{J1df}
The set of states $\mathcal{P}_{\rm exp}(\Gamma)$ is defined as that
containing all those states $\mu\in \mathcal{P}(\Gamma)$ for which
$B_\mu$ can be continued, as a function of $\theta$, to an
exponential type entire function on $L^1 (\mathbb{R}^d)$.
\end{definition}
It can be shown that a given $\mu$ belongs to $\mathcal{P}_{\rm
exp}(\Gamma)$ if and only if its functional $B_\mu$ can be written
down in the form
\begin{eqnarray}
  \label{I3}
B_\mu(\theta) = 1+ \sum_{n=1}^\infty
\frac{1}{n!}\int_{(\mathbb{R}^d)^n} k_\mu^{(n)} (x_1 , \dots , x_n)
\theta (x_1) \cdots \theta (x_n) d x_1 \cdots d x_n,
\end{eqnarray}
where $k_\mu^{(n)}$ is the $n$-th order correlation function of
$\mu$. It is a symmetric element of $L^\infty ((\mathbb{R}^d)^n)$
for which
\begin{equation}
\label{I4}
  \|k^{(n)}_\mu \|_{L^\infty
((\mathbb{R}^d)^n)} \leq C \exp( \vartheta n), \qquad n\in
\mathbb{N}_0,
\end{equation}
with some $C>0$ and $\vartheta \in \mathbb{R}$. Note that
$k_{\pi_\varkappa}^{(n)} (x_1 , \dots , x_n)= \varkappa^n$. Note
also that (\ref{I3}) resembles the Taylor
expansion of the characteristic function of a probability measure.
In view of this, $k^{(n)}_\mu$ are also called (factorial)
\emph{moment functions}, cf. \cite{BP1,BP3,Mu}.
\begin{remark}
  \label{I1rk}
By (\ref{I4}) each $\mu\in \mathcal{P}_{\rm exp}(\Gamma)$ satisfies
(\ref{J2}) and hence is a sub-Poissonian state.
\end{remark}
Recall that $\Gamma_0$ -- the set of all finite $\gamma \in \Gamma$ defined in
(\ref{J3a}) -- is an element of $\mathcal{B}(\Gamma)$. A function $G:\Gamma_0
\to \mathds{R}$ is $\mathcal{B}(\Gamma)/\mathcal{B}(\mathds{R}
)$-measurable, see \cite{FKKK}, if and only if, for each $n\in
\mathds{N}$, there exists a symmetric Borel function $G^{(n)}:
(\mathds{R}^{d})^{n} \to \mathds{R}$ such that
\begin{equation}
 \label{7}
 G(\eta) = G^{(n)} ( x_1, \dots , x_{n}), \quad {\rm for} \ \eta = \{ x_1, \dots , x_{n}\}.
\end{equation}
\begin{definition}
  \label{Gdef}
A measurable function $G:\Gamma_0 \to \mathds{R}$  is said to have bounded
support if: (a) there exists $\Lambda \in \mathcal{B}_{\rm b}
(\mathds{R}^d)$ such that $G(\eta) = 0$ whenever $\eta\cap
(\mathds{R}^d \setminus \Lambda)\neq \emptyset$; (b) there exists
$N\in \mathds{N}_0$ such that $G(\eta)=0$ whenever $|\eta|
>N$.  By $\Lambda(G)$ and $N(G)$ we denote the
smallest $\Lambda$ and $N$ with the properties just mentioned. By
$B_{\rm bs}(\Gamma_0)$ we denote the set of all such functions.
\end{definition}
The Lebesgue-Poisson measure $\lambda$ on $(\Gamma_0,
\mathcal{B}(\Gamma_0))$ is defined by the following formula
\begin{eqnarray}
\label{8} \int_{\Gamma_0} G(\eta ) \lambda ( d \eta)  = G(\emptyset)
+ \sum_{n=1}^\infty \frac{1}{n! } \int_{(\mathds{R}^d)^{n}} G^{(n)}
( x_1, \dots , x_{n} ) d x_1 \cdots dx_{n},
\end{eqnarray}
holding for all $G\in B_{\rm bs}(\Gamma_0)$. Like in (\ref{7}), we
introduce $k_\mu : \Gamma_0 \to \mathds{R}$ such that $k_\mu(\eta) =
k^{(n)}_\mu (x_1, \dots , x_n)$ for $\eta = \{x_1, \dots , x_n\}$,
$n\in \mathds{N}$. We also set $k_\mu(\emptyset)=1$. With the help
of the measure introduced in (\ref{8}), the expressions for $B_\mu$ in
(\ref{I1}) and (\ref{I3}) can be combined into the following formulas
\begin{eqnarray}
  \label{1fa}
 B_\mu (\theta)& = & \int_{\Gamma_0} k_\mu(\eta) \prod_{x\in \eta} \theta (x) \lambda (d\eta)=: \int_{\Gamma_0} k_\mu(\eta) e( \eta; \theta) \lambda (d \eta)
 \\[.2cm]
 & = &  \int_{\Gamma} \prod_{x\in \gamma} (1+ \theta (x)) \mu (d \gamma) =: \int_{\Gamma} F_\theta (\gamma) \mu(d
 \gamma). \nonumber
\end{eqnarray}
Thereby, one can transform  the action of $L$ on $F$, see
(\ref{L}), to the action of $L^\Delta$ on $k_\mu$ according to the
rule
\begin{equation}
  \label{1g}
\int_{\Gamma}(L F_\theta) (\gamma) \mu(d \gamma) = \int_{\Gamma_0}
(L^\Delta k_\mu) (\eta) e(\eta;\theta)
 \lambda (d \eta).
\end{equation}
This will allow us to pass from (\ref{R1}) to the corresponding
Cauchy problem for the correlation functions
\begin{equation}
  \label{J7}
\frac{d}{dt} k_t = L^\Delta k_t, \qquad k_t|_{t=0} = k_{\mu_0}.
\end{equation}
For the Bolker-Pacala model specified in (\ref{L}) and (\ref{J4}),
(\ref{J6}), by (\ref{1g}) the action of $L^\Delta$ looks as follows,
cf. \cite{FKKK,KK},
\begin{eqnarray}
  \label{J8}
\left( L^\Delta k \right) (\eta)& = & (L^{\Delta,-}k)(\eta) + \sum_{x\in\eta} E^{+} (x , \eta\setminus x) k(\eta \setminus x)\\[.2cm] \nonumber
& + &
\int_{\mathds{R}^d} \sum_{x\in \eta} a^{+} (x-y) k(\eta \setminus x
\cup y) d y, \nonumber
\end{eqnarray}
where
\begin{equation}
 \label{J8a}
(L^{\Delta,-}k)(\eta)  :=
- E^{-}(\eta) k(\eta) -
\int_{\mathds{R}^d}\left(\sum_{y\in \eta} a^{-}(x-y) \right) k(\eta\cup x)d x,
 \end{equation}
and
\begin{equation}
  \label{J9}
  E^{-}(\eta) := \sum_{x\in \eta} E^{-}(x ,\eta \setminus x).
\end{equation}
For the migration model specified in (\ref{L}) and (\ref{J5}),
(\ref{J6}), by (\ref{1g}) we obtain
\begin{equation}
  \label{J10}
\left( L^\Delta k \right) (\eta) = (L^{\Delta,-}k)(\eta) +
\sum_{x\in \eta} b(x) k(\eta \setminus x),
\end{equation}
with the same $L^{\Delta,-}$ as in (\ref{J8a}).  In the next
subsection, we introduce the spaces where we are going to define the
problems (\ref{J7}).

\subsection{The statements}
By (\ref{I3}) and (\ref{1fa}), it follows that $\mu \in
\mathcal{P}_{\rm exp}(\Gamma)$ implies
\begin{equation*}
 |k_\mu (\eta)| \leq C \exp( \vartheta  |\eta|),
\end{equation*}
holding for $\lambda$-almost all $\eta\in \Gamma_0$, some $C>0$, and
$\vartheta\in \mathds{R}$. In view of this, we set
\begin{equation}
 \label{18}
\mathcal{K}_\vartheta := \{ k:\Gamma_0\to \mathds{R}:
\|k\|_\vartheta <\infty\},
\end{equation}
where
\begin{equation}
  \label{17a}
 \|k\|_\vartheta = \esssup_{\eta \in \Gamma_0}\left\{ |k_\mu (\eta)| \exp\big{(} - \vartheta
  |\eta| \big{)} \right\}.
\end{equation}
Clearly, (\ref{18}) and (\ref{17a}) define a Banach space. In the following, we use the ascending scale of such
spaces $\mathcal{K}_\vartheta$, $\vartheta \in \mathds{R}$, with the
property
\begin{equation}
  \label{19}
\mathcal{K}_\vartheta \hookrightarrow \mathcal{K}_{\vartheta'},
\qquad \vartheta < \vartheta',
\end{equation}
where $\hookrightarrow$  denotes continuous embedding.

For $G\in B_{\rm
bs}(\Gamma)$, we set
\begin{equation}
  \label{9a}
(KG)(\gamma) = \sum_{\eta \Subset \gamma} G(\eta),
\end{equation}
where $\Subset$ denotes the summation is taken over all finite subsets. It satisfies,
see Definition \ref{Gdef},
\begin{equation*}
|(KG)(\gamma)| \leq \left( 1 + |\gamma\cap\Lambda (G)|\right)^{N(G)}.
\end{equation*}
The latter means that $\mu(KG) < \infty$ for each $\mu\in
\mathcal{P}_{\rm exp}(\Gamma)$. By (\ref{1fa}) this yields
\begin{equation}
 \label{J7a}
\langle \! \langle G, k_\mu \rangle \! \rangle := \int_{\Gamma_0}
G(\eta) k_\mu(\eta) \lambda (d \eta) =\mu(KG)
  < \infty.
\end{equation}
Set
\begin{equation}
  \label{9g}
B^\star_{\rm bs} (\Gamma_0) =\{ G\in B_{\rm bs}(\Gamma_0):
(KG)(\gamma) \geq 0 \ {\rm for} \ {\rm all} \ \gamma\in \Gamma\}.
\end{equation}
By \cite[Theorems 6.1 and 6.2 and Remark 6.3]{Tobi} one can prove
the next statement.
\begin{proposition}
  \label{Gpn}
Let  a measurable function $k : \Gamma_0 \to \mathds{R}$  have the
following properties:
\begin{eqnarray}
  \label{9h}
& (a) & \ \langle \! \langle G, k \rangle \!\rangle \geq 0, \qquad
{\rm for} \ {\rm all} \ G\in B^\star_{\rm bs} (\Gamma_0);\\[.2cm]
& (b) & \ k(\emptyset) = 1; \qquad (c) \ \ k(\eta) \leq
 C^{|\eta|} ,
\nonumber
\end{eqnarray}
with (c) holding for some $C >0$ and $\lambda$-almost all $\eta\in
\Gamma_0$. Then there exists a unique state $\mu \in
\mathcal{P}_{\rm exp}(\Gamma)$ for which $k$ is the correlation
function.
\end{proposition}
Set, cf. (\ref{9g}),
\begin{equation}
  \label{19a}
\mathcal{K}^\star_\vartheta =\{k\in \mathcal{K}_\vartheta: \langle
\! \langle G,k \rangle \! \rangle \geq 0 \ {\rm for} \ {\rm all} \
G\in B^\star_{\rm bs} (\Gamma_0)\},
\end{equation}
which is a subset of the cone
\begin{equation}
  \label{19b}
\mathcal{K}^+_\vartheta =\{k\in \mathcal{K}_\vartheta: k(\eta) \geq
0 \ \ {\rm for} \  \lambda-{\rm almost} \ {\rm all} \ \eta \in
\Gamma_0\}.
\end{equation}
By Proposition \ref{Gpn} it follows that each $k\in
\mathcal{K}^\star_\vartheta$ such that $k(\emptyset) = 1$ is the
correlation function of a unique state $\mu\in \mathcal{P}_{\rm
exp}(\Gamma)$. Then we define
\begin{equation*}
\mathcal{K} = \bigcup_{\vartheta \in \mathds{R}}
\mathcal{K}_\vartheta, \qquad \mathcal{K}^\star = \bigcup_{\vartheta
\in \mathds{R}} \mathcal{K}_\vartheta^\star.
\end{equation*}
As a sum of Banach spaces, the linear space $\mathcal{K}$ is
equipped with the corresponding inductive topology that turns it
into a locally convex space.

For each $\vartheta \in \mathds{R}$ and $\vartheta'>\vartheta$, the
expressions in (\ref{J8}), (\ref{J8a}) and (\ref{J10}) can be used to
define the corresponding bounded linear operators
$L^\Delta_{\vartheta' \vartheta}$ acting from
$\mathcal{K}_\vartheta$ to $\mathcal{K}_{\vartheta'}$. Their
operator norms can be estimated similarly as in \cite[eqs. (3.11),
(3.13)]{KK}, which yields, cf. (\ref{J6a}),
\begin{eqnarray}
  \label{J11}
 &(a) &   {\rm Bolker-Pacala} \ {\rm model\!:} \nonumber \\[.1cm]
& & \qquad \qquad \qquad  \|L^\Delta_{\vartheta' \vartheta}\| \leq
\frac{4 (\| a^{+} \| +\| a^{-} \|  )}{e^2 (\vartheta' -
\vartheta)^2} + \frac{ \langle a^{+} \rangle  + \langle a^{-}
\rangle
 e^{\vartheta'}}{e (\vartheta' - \vartheta)}, \qquad
\qquad \qquad  \\[.3cm]
 &(b) &  {\rm migration} \ {\rm model\!:} \nonumber \\[.1cm]
\label{J12} & & \qquad  \qquad \qquad  \|L^\Delta_{\vartheta'
\vartheta}\| \leq  \frac{4 \| a^{-} \| }{e^2 (\vartheta' -
\vartheta)^2} + \frac{ \|b\| e^{-\vartheta} + \langle a^{-} \rangle
e^{\vartheta'}}{e (\vartheta' - \vartheta)}. \qquad \qquad \qquad
\end{eqnarray}
By means of the collection $\{L^\Delta_{\vartheta'
\vartheta}\}$ we introduce the corresponding continuous linear
operators acting on $\mathcal{K}$, and thus define the
corresponding Cauchy problems (\ref{J7}) in this space. By their
(global in time) solutions we will mean continuously differentiable
functions $[0,+\infty) \ni t \mapsto k_t \in \mathcal{K}$ such that
both equalities in (\ref{J7}) hold. Our results are given in the
following statements, both based on Assumption \ref{Ass1}.
\begin{theorem}[Bolker-Pacala model]
  \label{1tm}
For each $\mu_0 \in \mathcal{P}_{\rm exp} (\Gamma)$, the problem in
(\ref{J7}) with $L^\Delta:\mathcal{K}\to \mathcal{K}$ as in
(\ref{J8}) -- (\ref{J9}) and (\ref{J11}) has a unique solution which
lies in  $\mathcal{K}^\star$ and is such that $k_t(\emptyset)=1$ for
all $t>0$. Therefore, for each $t>0$, there exists a unique state
$\mu_t\in \mathcal{P}_{\rm exp}(\Gamma)$ such that $k_t =
k_{\mu_t}$.
\end{theorem}
\begin{theorem}[Migration model]
  \label{2tm}
For each $\mu_0 \in \mathcal{P}_{\rm exp} (\Gamma)$, the problem in
(\ref{J7}) with $L^\Delta:\mathcal{K}\to \mathcal{K}$ as in
(\ref{J8a}), (\ref{J10}) and (\ref{J12}) has a unique solution which
lies in  $\mathcal{K}^\star$ and is such that $k_t(\emptyset)=1$ for
all $t>0$. Therefore, for each $t>0$, there exists a unique state
$\mu_t\in \mathcal{P}_{\rm exp}(\Gamma)$ such that $k_t =
k_{\mu_t}$. Moreover, these states $\mu_t$ have the property: for
every $n\in \mathds{N}$ and $\Lambda \in \mathcal{B}_{\rm
b}(\mathds{R}^d)$, the following holds, cf. (\ref{J3b}),
\[
\forall t>0 \qquad  \mu_t (N_\Lambda^n) \leq C_\Lambda^{(n)}.
\]
\end{theorem}

\subsection{Comments and comparison}

For $a^{-} \equiv 0$, both models considered in this work get
exactly soluble. The Bolker-Pacala version with $E^{-}(x,\gamma) =
m$  ($m\geq 0$ being the intrinsic mortality rate) is known as the
continuum contact model -- see \cite{Dima}, and also the discussion
in \cite[Introduction]{KK} and the literature cited therein. In this
model, the evolution $\mu_0 \mapsto \mu_t$ does not preserve the
class $\mathcal{P}_{\rm exp}(\Gamma)$, cf. \cite[eq. (3.5), page
303]{Dima}. For $m\geq \langle a^{+} \rangle$, the correlation
functions remain bounded in time. That is, the global
regulation is achieved at the expense of large intrinsic
mortality. Moreover, the system dies out if $m> \langle a^{+}
\rangle$.

The migration version of (\ref{L}) with $E^{-}(x,\gamma) = m$ is
known as the Surgailis model, see \cite{Sur} and the discussion in
\cite{DimaR}. See also its solution in (\ref{A53}) below obtained
for $E^{-} \equiv 0$. For this model with $m>0$, the large-time
limits of the correlation functions are Poissonian with the density
$\varrho(x) = b(x)/m$. If the initial state is Poissonian with
density $\varrho_0 (x)$,  and if $m=0$, the state $\mu_t$ is also
Poissonian with the density $\varrho_t (x) = \varrho_0 (x) + b(x)
t$, cf. (\ref{A53}) below. That is, in this model the global regulation is possible
only if $m(x) \geq m>0$ for all $x\in \mathds{R}^d$.

Now we give more specific comments to each of the models.

\subsubsection{The Bolker-Pacala model}

As follows from our Theorem \ref{1tm},  adding competition to the
continuum contact model mentioned above yields the local
self-regulation -- no matter how long the dispersal is  and how
local is the competition. Also their magnitudes do not matter for
the very fact of the self-regulation. A sufficient condition under
which the property stated in Theorem \ref{1tm} (as well as in
\cite[Theorem 3.3]{KK}) holds, cf. \cite[eq. (3.5)]{KK}, in the
present notations is
\begin{equation}
  \label{J13}
\omega |\eta| + \sum_{x\in \eta}\sum_{y\in \eta\setminus x}
a^{-}(x-y) \geq \theta \sum_{x\in \eta}\sum_{y\in \eta\setminus x}
a^{+}(x-y),
\end{equation}
holding for some $\omega \geq 0$ and $\theta>0$, and
$\lambda$-almost all $\eta \in\Gamma_0$. Recall that $|\eta|$ stands
for the cardinality of $\eta\in \Gamma_0$. In the short dispersal
case (studied in \cite{FKKK}) the condition in (\ref{J13}) readily
holds with $\omega =0$. Then the most intriguing question here is
whether it can hold in the long dispersal case. In \cite[Proposition
3.7]{KK}, it was shown that measurable $a^{+}$ and $a^{-}$ satisfy
(\ref{J13}) with some $\omega$ and $\theta$ if $a^{-}(x)$ is
separated away from zero for $|x|<r$ with some $r>0$, and
$a^{+}(x)\equiv 0$ for $|x| \geq R$ with some $R>0$ with the
possibility $R>r$. Another choice of $a^{+}$ and $a^{-}$ satisfying
(\ref{J13}) can be, see \cite[Proposition 3.8]{KK},
\[
a^{\pm} (x) = \frac{c_{\pm}}{( 2 \pi\sigma_{\pm}^2)^d/2} \exp\left(
- \frac{1}{2 \sigma_{\pm}^2} |x|^2\right),
\]
with all possible values of the parameters $c_{\pm} >0$ and
$\sigma_{\pm} >0$. An important example of $a^{\pm}$ which both
Propositions 3.7 and 3.8 of \cite{KK} do not cover is
$a^{-}$ having finite range and $a^{+}$ being Gaussian as above. The
novelty of our present (rather unexpected) result is that
(\ref{J13}) is satisfied \emph{for any} $a^{+}$ and $a^{-}$ as in
Assumption \ref{Ass1}, and hence the local self-regulation is
achieved by applying any kind of competition. Does not matter how
weak and short-ranged. Finally, we remark that for this model the
conditions of boudedness and continuity of $a^{-}$ can be relaxed.
As follows from the proof of Lemma \ref{J1lm} below, like in
\cite{KK} it is enough to assume that $a^{-}$ is measurable and
separated away from zero in some ball. As for $a^{+}$, it is enough to
have a continuous $\tilde{a}^{+}\in L^\infty (\mathds{R}^d)\cap
L^1(\mathds{R}^d)$ such that $\tilde{a}^{+}(x) \geq a^{+}(x)$ for
almost all $x$.

\section{The Proof Theorems \ref{1tm} }

In \cite[Theorem 3.3]{KK}, it was proved that the evolution in
question exists whenever the kernels $a^{\pm}$ are bounded and
integrable, and satisfy (\ref{J13}) with some $\omega$ and $\theta$.
Thus, the proof of Theorem \ref{1tm} relies upon proving the
following statement.
\begin{lemma}
  \label{J1lm}
Let $a^{\pm}$ satisfy Assumption \ref{Ass1}. Then one finds $\omega
\geq 0$ and $\theta >0$ such that (\ref{J13}) holds for all $\eta\in
\Gamma_0$.
\end{lemma}
\begin{proof}
Since $a^{+}$ is Riemann integrable, for an arbitrary $\varepsilon>0$, one
can divide $\mathds{R}^d$ into equal cubic cells $E_l$, $l\in \mathds{N}$,
of small enough side $h>0$ such that the following holds, see
(\ref{J6a}),
\begin{equation}
  \label{J14}
 h^d \sum_{l=1}^\infty a^{+}_l  \leq \langle a^{+} \rangle +
  \varepsilon, \qquad a^{+}_l := \sup_{x\in E_l} a^{+}(x).
\end{equation}
Given $r>0$ and $x\in \mathds{R}^d$, we set $K_r (x) = \{y\in
\mathds{R}^d: |x-y|<r\}$ and
\begin{equation}
\label{J14a} a^{-}_r = \inf_{x\in
  K_{2r}(0)} a^{-} (x).
\end{equation}
Then we fix $\varepsilon$ and pick $r>0$ such that $a^{-}_r>0$. In
the following, $r$, $h$ and $\varepsilon$ are fixed.

The proof of the lemma will be done by the induction in the number
of points in $\eta$. To do this we rewrite (\ref{J13}) in the form
\begin{equation}
  \label{J16}
U_\theta (\eta) := \omega |\eta| + \sum_{x\in \eta}\sum_{y\in
\eta\setminus x} a^{-}(x-y) - \theta \sum_{x\in \eta}\sum_{y\in
\eta\setminus x} a^{+}(x-y) \geq 0.
\end{equation}
For some $x\in \eta$, consider
\begin{eqnarray}
  \label{P2}
U_\theta (x, \eta\setminus x) & := & U_\theta (\eta) -
U_\theta (\eta \setminus x) \\[.2cm] \nonumber & = &
\omega + 2 \left(\sum_{y \in \eta \setminus x} a^{-} (x-y) -
\theta\sum_{y \in \eta \setminus x} a^{+} (x-y) \right).
\end{eqnarray}
Let $c_d$ be the volume of the unit ball $K_1$ and $\Delta (d)$ be the
packing constant for the rigid balls in $\mathds{R}^d$, cf. \cite{Groemer}.
Set
\begin{gather}
  \label{J17}
  \delta (a^{+}) = \max\left\{ \|a^{+}\|; (\langle a^{+} \rangle +
  \varepsilon)g_d(h,r) \right\}, \\[.2cm] \nonumber g_d(h,r)= \frac{\Delta(d)}{c_d}\left(\frac{h+2r}{hr}
  \right)^d,
\end{gather}
and assume that $\omega$ and $\theta$ satisfy the following, cf.
(\ref{J14a}),
\begin{equation}
  \label{J18}
 \theta \leq \min\left\{\frac{\omega}{2 \delta(a^{+})}; \frac{a^{-}_r}{\delta(a^{+})}
 \right\}.
\end{equation}
Let us show that:
\begin{itemize}
  \item[(a)] for each $\eta=\{x,y\}$, (\ref{J18}) yields (\ref{J16});
\item[(b)] for each $\eta$, one finds $x\in \eta$ such that
$U_\theta  (x, \eta \setminus x) \geq 0$ whenever (\ref{J18}) holds.
\end{itemize}
If both (a) and (b) hold, then (\ref{J16}) will follow from
(\ref{P2}) by the induction in $|\eta|$. To prove (a) we write
\begin{gather*}
U_\theta (\{x,y\}) = 2\omega + 2 a^{-} (x-y) - 2 \theta a^{+} (x-y) \\[.2cm]
\geq \left(\omega - 2 \theta \| a^{+}\| \right) + 2
a^{-} (x-y) \geq 0,
\end{gather*}
with the latter estimate following by (\ref{J18}) and (\ref{J17}). To prove
(b), for $y\in \eta$, we set
\begin{equation}
  \label{J19}
  s= \max_{y\in \eta}
  \left\vert\eta \cap K_{2r}(y)\right\vert.
\end{equation}
Let also $x\in \eta$ be such that $\left\vert\eta \cap K_{2r}(x)\right\vert=s$.
For this $x$, by
$E_l(x)$, $l\in \mathds{N}$, we denote the corresponding translates
of $E_l$ which appear in (\ref{J14}). Set $\eta_l = \eta \cap E_l
(x)$ and let $l_*\in \mathds{N}$ be such that $\eta \subset
\cup_{l\leq l_*} E_l (x)$, which is possible since $\eta$ is finite.
For a given $l$, a subset $\xi_l \subset \eta_l$ is called
$r$-admissible if for each distinct $y,z \in \xi_l$, one has that
$K_r(y) \cap K_r(z) = \emptyset$. Such a subset $\xi_l$ is called
maximal $r$-admissible if $|\xi_l| \geq |\xi'|$ for any other
$r$-admissible $\xi'_l$. It is clear that
\begin{equation}
  \label{J20}
\eta_l \subset \bigcup_{z\in \xi_l} K_{2r}(z).
\end{equation}
Otherwise, one finds $y\in \eta_l$ such that $|y-z|\geq 2r$, for
each $z\in \xi_l$, which yields that $\xi_l$ is not maximal. Since
all the balls $K_r (z)$, $z\in \xi_l$, are contained in the $h$-extended cell
$$E_l^h(x):=\{y\in \mathds{R}^d: \inf_{ z\in E_l(x)} |y-z|\leq h \},$$
their maximum number -- and hence $|\xi_l|$ -- can be estimated as
follows
\begin{equation}
  \label{J15}
|\xi_l| \leq \Delta (d) {\rm V}(E_l^h (x))/c_d r^d = h^d
\frac{\Delta (d)}{c_d}\left( \frac{h+2r}{h r} \right)^d = h^d
g_d(h,r),
\end{equation}
where $c_d$ and $\Delta(d)$ are as in (\ref{J17}).
Then by
(\ref{J19}) and (\ref{J20}) we get
\begin{gather*}
 \sum_{y\in\eta\setminus x} a^{+} (x-y) \leq
 \sum_{l=1}^{l_*} \sum_{z\in \xi_l} \sum_{y\in K_{2r}(z) \cap
 \eta_l}a_l^{+}.
\end{gather*}
The cardinality of $K_{2r}(z) \cap
 \eta_l$ does not exceed $s$, see (\ref{J19}), whereas the
 cardinality of $\xi_l$ satisfies (\ref{J15}). Then
\begin{gather}
\label{J21}
 \sum_{y\in\eta\setminus x} a^{+} (x-y) \leq
s g_d(h,r) \sum_{l=1}^\infty a^{+}_l h^d \leq s g_d(h,r) (\langle
a^{+} \rangle + \varepsilon) \leq s \delta(a^{+}).
\end{gather}
On the other hand, by (\ref{J14a}) and (\ref{J19}) we get
\[
\sum_{y\in \eta\setminus x} a^{-}(x-y) \geq \sum_{y\in
(\eta\setminus x)\cap K_{2r}(x)} a^{-}(x-y) \geq (s-1) a^{-}_r.
\]
We use this estimate and (\ref{J21}) in (\ref{P2}) and obtain
\[
U_\theta (x, \eta \setminus x) \geq 2 \delta (a^{+}) \left[
 \left(\frac{\omega}{2 \delta (a^{+})} - \theta\right) + (s-1)
\left(\frac{a_r^{-}}{\delta (a^{+})} - \theta \right) \right] \geq
0,
\]
see (\ref{J18}). Thus, claim (b) also holds, which completes the
proof.
\end{proof}
Now the proof of Theorem \ref{1tm} follows by \cite[Theorem 3.3]{KK} and Lemma \ref{J1lm} just proved.

\section*{Acknowledgment}
The present research was supported by the European Commission under
the project STREVCOMS PIRSES-2013-612669 and by the SFB 701
``Spektrale Strukturen and Topologische Methoden in der Mathematik".

\end{document}